\documentclass[a4j,12pt,reqno]{amsart}

\usepackage{amscd,epsfig,amsthm,graphicx}
\usepackage{amssymb}
\usepackage{amsmath}
\usepackage{amsthm}
\usepackage{amsfonts}
\usepackage[english]{babel}
\usepackage{epic,eepic}
\usepackage[flushleft]{paralist} 
\usepackage[utf8]{inputenc}
\usepackage{enumerate}
\usepackage{graphicx}
\usepackage[all]{xy} 

\makeindex			 
\newtheorem{theorem}{Theorem}
\newtheorem{lemma}[theorem]{Lemma}

\theoremstyle{definition}
			{Definition}

\newtheorem{remark}{Remark}

\makeatletter
    
    \@addtoreset{equation}{section}
\makeatother

\def\N{{\mathbb N}}
\def\Z{{\mathbb Z}} \def\Q{{\mathbb Q}}
\def\C{{\mathbb C}} \def\R{{\mathbb R}}

\makeatletter
\def\@setcopyright{}
\def\serieslogo@{}
\makeatother

\begin{document}

%


\author{Takashi MIYAGAWA}      %


\title{Approximate functional equation and upper bounds for the Barnes double zeta-function}

\begin{abstract}
 As one of the asymptotic formulas of the zeta-function, Hardy and Littlewood 
 gave asymptotic formulas called the approximate functional equation.
 In this paper, we prove an approximate functional equation of the 
 Barnes double zeta-function 
 $ \zeta_2 (s, \alpha ; v, w ) = \sum_{m=0}^\infty \sum_{n=0}^\infty (\alpha+vm+wn)^{-s} $.
 Also, applying this approximate functional equation and the van der Corput method, 
 we obtain upper bounds for 
 $ \zeta_2(1/2 + it, \alpha ; v, w) $ and $ \zeta_2(3/2 + it, \alpha ; v, w) $
 with respect to $ t $ as $ t \rightarrow \infty $.
\end{abstract}


\maketitle

\dedicatory{}

\date{}


\maketitle


\section{Introduction and statement of results}
Let $ s = \sigma + it $ be a complex variable, and let 
$ \alpha > 0 $ and $ v, w > 0 $ be real parameters.

In this section, we introduce the approximate functional equations of the Riemann
zeta-function
\[
	\zeta(s) = \sum_{n=1}^\infty \frac{1}{n^s},
\]
some results of upper bounds for $ \zeta(1/2 + it) $ with respect to $ t \rightarrow \infty $,
Barnes multiple zeta-function, and give the main theorems on 
the approximate functional equation for the Barnes double zeta-function
\begin{equation}\label{B_2-zeta}
		\zeta_2 (s, \alpha ; v, w)
               = \sum_{m = 0}^\infty \sum_{n = 0}^\infty
                 \frac{1}{(\alpha + v m + w n)^s}.
\end{equation}

As a classical asymptotic formula of the Riemann zeta-function,
the following formula was proved by Hardy and Littlewood;
\[
		\zeta(s) = \sum_{n \leq x} \frac{1}{n^s} - \frac{x^{1-s}}{1-s} 
		+ O(x^{-\sigma}) \qquad (x \rightarrow \infty),
\]
uniformly for $ \sigma \geq \sigma_0 > 0,\ |t| < 2\pi x/C $, when $ C > 1 $ 
is a constant. This formula gives an indication in the discussion in the
critical strip of $ \zeta(s) $.
Also, Hardy and Littlewood proved the following asymptotic formula 
({\S}4 in \cite{Tit}); suppose that 
$ 0 \leq \sigma \leq 1,\ x \geq 1,\ y\geq 1 $ and $ 2\pi ixy = |t| $ then
\begin{equation}\label{AFE}
	\zeta(s) = \sum_{n \leq x} \frac{1}{n^s} + \chi(s) \sum_{n \leq y} \frac{1}{n^{1-s}}
	+ O(x^{-\sigma}) + O(|t|^{1/2-\sigma}y^{\sigma-1}),
\end{equation}
where $ \chi(s) = 2 \Gamma(1-s) \sin{(\pi s/2)}(2\pi)^{s-1} $ and note that the functional equation
$ \zeta(s) = \chi(s) \zeta(1-s) $ holes. This formula (\ref{AFE}) is called the
approximate functional equation.

The order of $ |\zeta(\sigma + it)| $ with respect to $ t $ is an extremely important problem
in the deeper theory of the Riemann zeta-fucntion $ \zeta(s) $. In particular, the order of 
$ |\zeta(1/2 + it)| $ is the most important. For example, Hardy-Littlewood improve to
\begin{equation}\label{Hardy-Littlewood}
	\zeta \left( \frac12 + it \right) \ll t^{1/6 + \varepsilon}
\end{equation}
by the van der Corput method, applying to (\ref{AFE}).
In 1988, $ \zeta(1/2 + it) \ll t^{9/56 + \varepsilon} $ was proved by Bombieri and Iwaniec,
after which many mathematicians gradually improved, and now 
$ \zeta(1/2 + it) \ll t^{32/205 + \varepsilon} $ has been proved by Huxley in 2005.
Furthermore in 2017, $ \zeta(1/2 + it) \ll t^{13/84 + \varepsilon} $ was proved by
Bourgain (in see \cite{Bou}).

Let $ r $ be a positive integer and $ w_j > 0 \ (j = 1, \ldots , r) $ are complex parameters.
The Barnes multiple zeta-function $ \zeta_r(s, \alpha ; w_1, \ldots , w_r) $ is defined by
\begin{equation}\label{B-zeta}
       	\zeta_r (s, \alpha ; w_1, \ldots , w_r)
               = \sum_{m_1 = 0}^\infty \cdots \sum_{m_r = 0}^\infty
                 \frac{1}{(\alpha + w_1 m_1 + \cdots + w_r m_r)^s}
\end{equation}
where the series on the right hand-side is 
absolutely convergence for $ \mathrm{Re}(s) > r $,
and is continued meromorphically to $ \C $ and its only singularities are 
the simple poles located at $ s = j \ (j = 1, \ldots , r) $.
This series is a multiple version of the Hurwitz zeta-function
\begin{equation}
	\zeta_H(s, \alpha) = \sum_{n=0}^\infty \frac{1}{(n+\alpha)^s}
	\qquad (0 < \alpha \leq 1).
\end{equation}
Also, as a generalization of this series in another direction, the Lerch zeta-function
\begin{equation}\label{zeta_L}
	\zeta_L(s, \alpha, \lambda) = \sum_{n=0}^\infty \frac{e^{2\pi in \lambda}}{(n+\alpha)^s}
	\qquad (0 < \alpha \leq 1, \ 0 < \lambda \leq 1)
\end{equation}
is also an important research subject. These series are absolutely convergent for $ \sigma > 1 $.
Also, if $ 0 < \lambda < 1 $, then the series (\ref{zeta_L}) is convergent even for $ \sigma > 0 $.


\begin{remark}
The Barnes double zeta-function was introduced by E. W. Barnes \cite{B1}
in the theory of double gamma function, and double series of the 
form (\ref{B_2-zeta}) is introduced in \cite{B2}. Furthermore in \cite{B3}, 
in connection with the theory of the multiple gamma function, and 
multiple series of the form (\ref{B-zeta}) was introduced.
\end{remark}

\bigskip
\begin{theorem}[Theorem 3 in \cite{TM2}]\label{order_of_Barnes_zeta}
Let $ 0 < \sigma_1 < \sigma_2 < 2,\ x \geq 1 $ and $ C > 1 $.
Suppose $ s = \sigma + it \in \C $ with $ \sigma_1 < \sigma < \sigma_2 $
and $ |t| \leq 2\pi x/C $. Then
\begin{eqnarray}
	&& \zeta_2(s,\alpha;v,w) = \mathop{\sum\sum}\limits_{0 \leq m,n \leq x} 
				\frac{1}{(\alpha + vm + wn)^s}		\nonumber \\
	&& \qquad\qquad	+ \frac{(\alpha + vx)^{2-s} + (\alpha + wx)^{2-s} - (\alpha + vx + wx)^{2-s}}
				{vw(s-1)(s-2)}  + O(x^{1-\sigma} )	  
\end{eqnarray}
as $ x \rightarrow \infty $. 
\end{theorem}

We prove an analogue of the approximate functional equation (\ref{AFE}) for (\ref{B_2-zeta})
 (in Theorem \ref{th:Main_Theorem1}).
In the following theorem, the results when the complex parameter $ v,w $
linearly independent are different from the results when $ v, w $ are linearly dependent 
over $ \Q $.
Also, we consider the upper bounds of $ \zeta_2(1/2 + it,\alpha;v,w) $ and 
$ \zeta_2(3/2 + it, \alpha; v,w) $, and the following Theorem \ref{th:Main_Theorem2}
and Theorem \ref{th:Main_Theorem3} were obtaind. Proof of Theorem \ref{th:Main_Theorem2},
can be obtained by using Theorem \ref{th:Main_Theorem1}, and proof of 
Theorem \ref{th:Main_Theorem3} can be obtained by using Theorem \ref{order_of_Barnes_zeta}.

\begin{theorem}\label{th:Main_Theorem1}
Suppose that $ 0 \leq \sigma \leq 2, x = x(t) \geq 1, y=y(t) \geq x(t) $ and $ 2\pi xy = |t| $. 
Let $ L,M,N $ are non-negative integer as satisfying 
$ N = \left[ x/(v+w) \right] $ and $ L = [vy], M = [wy] $.
\begin{enumerate}
	\item[$(i)$]
	If $ v, w $ are linearly independent over $ \mathbb{Q} $;
	\begin{eqnarray}
		&& \zeta_2(s, \alpha; v,w)		\nonumber \\
		&& = \mathop{\sum\sum}\limits_{0 \leq m,n \leq N} 															\frac{1}{(\alpha + vm + wn)^s}
			+ \frac{1}{w^s}\sum_{m=0}^N \zeta_H^*(s,\alpha_{v,m}) 
			+ \frac{1}{v^s}\sum_{n=0}^N \zeta_H^*(s, \alpha_{w,n})	\nonumber \\
		&& \ \ - \frac{\Gamma(1-s)}{(2\pi i)^{1-s}e^{\pi is}}\left\{ \frac{1}{v^s} 
			\sum_{0 < |n| < L} \frac{e^{-2\pi in(\alpha+wN)/v}}
					{(e^{2\pi inw/v}-1)n^{1-s}} + \frac{1}{w^s}
			\sum_{0 < |n| < M} \frac{e^{-2\pi in(\alpha+vN)/w}}
					{(e^{2\pi inv/w}-1)n^{1-s}} \right\} \nonumber \\
		&& \ \ + O(x^{-\sigma}) 	\label{Thm1(i)}
		\end{eqnarray}
	\item[$(ii)$]
	If $ v, w $ are linearly dependent over $ \mathbb{Q} $,
	exist $ p, q \in \mathbb{N} $ such as $ pv=qw $ and $ (p,\;q)=1 $. Then we have
	\begin{eqnarray}
		&& \zeta_2(s, \alpha; v,w)		\nonumber \\
		&&  = \mathop{\sum\sum}\limits_{0 \leq m,n \leq N} 
			\frac{1}{(\alpha + vm + wn)^s}
			+ \frac{1}{w^s}\sum_{m=0}^N \zeta_H^*(s,\alpha_{v,m}) 
			+ \frac{1}{v^s}\sum_{n=0}^N \zeta_H^*(s, \alpha_{w,n})		\nonumber \\
		&& \ \ - \frac{\Gamma(1-s)}{(2\pi i)^{1-s}e^{\pi is}}
                	\left\{ \frac{1}{v^s} 
			\mathop{\sum_{0 < |n| < L}}\limits_{q \; \mid \hspace{-.25em}/ \, n} 
			\frac{e^{-2\pi in(\alpha+wN)/v}}
				{(e^{2\pi inw/v}-1)n^{1-s}} + \frac{1}{w^s}
			\mathop{\sum_{0 < |n| < M}}\limits_{p \; \mid \hspace{-.25em}/ \, n} 
			\frac{e^{-2\pi in(\alpha+vN)/w}}
				{(e^{2\pi inv/w}-1)n^{1-s}}  \right. \nonumber \\
		&& \qquad
                	+ \frac{q^{s-1}}{2\pi ipv^s}(1-s) \sum_{0 < |n| < M} \frac{e^{-2q\pi in\alpha/v}}{n^{2-s}} 
                        \nonumber \\
                && \qquad	
                	- \left. \left( \frac{\alpha q}{pv^2} + \frac{(p+q)N}{pv} + \frac{vp}{2q} + \frac{v}2 \right) 
                	\left( \frac{q}{v} \right)^{s-1} \sum_{0 < |n| < M} \frac{e^{-2q\pi in\alpha /v}}{n^{1-s}} 
                        	 	\vbox to 35pt{} \right\} \nonumber \\
                && \ \  + O(x^{-\sigma}),
		\end{eqnarray}
	\end{enumerate}
	where
	\begin{eqnarray*}
		&& \zeta_H^* (s, \alpha_{v,m}) := \zeta_H(s, \alpha_{v,m}) 
			- \sum_{n=0}^{N+n_{v,m}} \frac{1}{(n + \alpha_{v,m})^s},	\\
		&& \alpha_{v,m} := 
		\begin{cases}
			\left\{ \dfrac{vm+\alpha}{w} \right\}  & 
			\Bigl(\dfrac{vm+\alpha}{w} \notin \mathbb{N} \Bigr), \\
			\quad 1 & \Bigl(\dfrac{vm+\alpha}{w} \in \mathbb{N} \Bigr),
		\end{cases}	\\
		&& n_{v,m} := 
		\begin{cases}
			\left[ \dfrac{vm+\alpha}{w} \right] - 1 & \Bigl(\dfrac{vm+\alpha}{w} \geq 1 \Bigr), \\
			\quad 0 & \Bigl( 0 < \dfrac{vm+\alpha}{w} < 1 \Bigr).
		\end{cases}
	\end{eqnarray*}
        The definitions of $ \zeta_H^*(s, \alpha_{w,n})$ and $ \alpha_{w,n} $ are similar.
	
\end{theorem}

\bigskip
\bigskip


\begin{theorem}\label{th:Main_Theorem2}
If $ v,w $ are linearly independent over $ \Q $, then we have
\[
	\zeta_2\left(\frac12+it, \alpha ; v,w \right) \ll 
	\begin{cases}
		|t|^{1/6}			& (\alpha, v, w \ are \ lin. \ indep. \ over \ \Q),	\\
		|t|^{1/6} \log{|t|} & (\alpha, v, w \ are \ lin. \ dep. \ over \ \Q).
	\end{cases}
\]
If $ v,w $ are linearly dependent over $ \Q $, exist $ p,q \in \N $ such as $ pv=qw $ and $ (p,q)=1 $.
Then we have
\[
	\zeta_2\left(\frac12+it, \alpha ; v,w \right) = \kappa t + O(|t|^{1/6} \log{|t|})
\]
where $ \kappa = \kappa(t) $ is a constant with
\[
	0 < |\kappa(t)| < \frac{1}{2\pi p \sqrt{qv}} \left| \zeta_L \left( \frac12+it, 1, 1 - \frac{q}{v}\alpha \right) \right|.
\]
\end{theorem}

\medskip
\medskip

\begin{theorem}\label{th:Main_Theorem3}
\[
	\zeta_2 \left(\frac32+it, \alpha ; v,w \right) \ll |t|^{1/3}.
\]
\end{theorem}

\medskip

\section{Proof of theorem \ref{th:Main_Theorem1}}

In this section, we give the proof of Theorem \ref{th:Main_Theorem1}.

\medskip
\textbf{Proof of Theorem \ref{th:Main_Theorem1}}. 

Let $ N \in \N $ be sufficiently large. Then we consider
\begin{eqnarray*}
	&& \zeta_2(s, \alpha; v,w)		\\
	&& \quad =\sum_{m=0}^\infty \sum_{n=0}^\infty 
			\frac{1}{(\alpha + v m + w n)^s}	\\			
	&& \quad = \left( \sum_{m=0}^N \sum_{n=0}^N 
				+ \sum_{m=0}^N \sum_{n=N+1}^\infty + \sum_{m=N+1}^\infty \sum_{n=0}^N 
				+ \sum_{m=N+1}^\infty \sum_{n=N+1}^\infty \right)
				\frac{1}{(\alpha + v m + w n)^s}.
\end{eqnarray*}
Also, the second term of the above is 
\begin{eqnarray*}
	&& \sum_{m=0}^N \sum_{n=N+1}^\infty \frac{1}{(\alpha + v m + w n)^s} \\
	&& \quad= \sum_{m=0}^N \left( \sum_{n=-n_{v,\alpha}}^\infty - \sum_{n=-n_{v,\alpha}}^N \right)
			\frac{1}{w^s \{n + (\alpha + v m)/w \}^s}	\\			
	&& \quad= \frac1{w^s} \sum_{m=0}^N \left\{ \zeta_H (s, \alpha_{v,m}) - \sum_{n=-n_{v,\alpha}}^N 
        	\frac{1}{\{n + (\alpha + v m)/w \}^s}  \right\} \\
        && \quad= \frac1{w^s} \sum_{m=0}^N \zeta_H^* (s, \alpha_{v,m}).
\end{eqnarray*}
Similarly, the third term of the above, we have
\[
	\sum_{m=N+1}^\infty \sum_{n=0}^N \frac{1}{(\alpha + v m + w n)^s} 
        = \frac1{v^s} \sum_{n=0}^{N} \zeta_H^* (s, \alpha_{w,n}).
\]
Transform the fourth term on the right hand-side of the above equation to the
contour integral to obtain
\begin{eqnarray}\label{contour_of_Barnes}
	\sum_{m=N+1}^\infty \sum_{n=N+1}^\infty \frac{1}{(\alpha + v m + w n)^s}
	= \frac{\Gamma(1-s)}{2\pi i e^{\pi is}}
		\int_C \frac{z^{s-1}e^{-(\alpha + vN + wN)z}}{(e^{vz}-1)(e^{wz}-1)} dz
\end{eqnarray}
where $ C $ is the contour integral path that comes from $ +\infty $ to $ \varepsilon $ along 
the real axis, then continues along the circle of radius $ \varepsilon $ counter clockwise, and finally
goes from $ \varepsilon $ to $ +\infty $.

\medskip

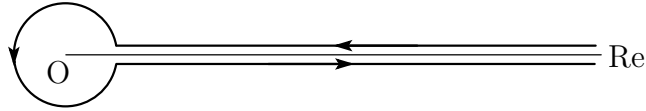
\begin{center}

{\unitlength 0.1in%
\begin{picture}( 30.9400,  5.4000)(  6.5000,-10.6000)%
%
\special{pn 13}%
\special{pa 2781 742}%
\special{pa 2361 742}%
\special{fp}%
\special{sh 1}%
\special{pa 2361 742}%
\special{pa 2428 762}%
\special{pa 2414 742}%
\special{pa 2428 722}%
\special{pa 2361 742}%
\special{fp}%
%
\special{pn 8}%
\special{pa 942 790}%
\special{pa 3744 790}%
\special{fp}%
%
\put(34.0000,-10.6000){\makebox(0,0)[lb]{}}%
\put(38.7500,-7.9500){\makebox(0,0){$ \mathrm{Re} $}}%
%
\special{pn 13}%
\special{pa 2002 838}%
\special{pa 2432 838}%
\special{fp}%
\special{sh 1}%
\special{pa 2432 838}%
\special{pa 2365 818}%
\special{pa 2379 838}%
\special{pa 2365 858}%
\special{pa 2432 838}%
\special{fp}%
\put(9.0000,-8.8000){\makebox(0,0){O}}%
%
\special{pn 13}%
\special{ar 940 790 270 270  0.1973956  6.1000745}%
%
\special{pn 13}%
\special{pa 670 778}%
\special{pa 670 828}%
\special{fp}%
\special{sh 1}%
\special{pa 670 828}%
\special{pa 690 761}%
\special{pa 670 775}%
\special{pa 650 761}%
\special{pa 670 828}%
\special{fp}%
\special{pa 670 828}%
\special{pa 670 828}%
\special{fp}%
%
\special{pn 13}%
\special{pa 1206 742}%
\special{pa 3715 742}%
\special{fp}%
%
\special{pn 13}%
\special{pa 1206 838}%
\special{pa 3710 838}%
\special{fp}%
\end{picture}}%

\end{center}

\medskip

Let $ \sigma \leq 2, t>0 $ and $ 1 \leq x < y $, so that 
$ 1 \leq x \leq \sqrt{t/2\pi} $.
Let $ L, M, N $ be non-negative integer satisfying
\[
	N = \left[ \frac{x}{v+w} \right], \ \ 
	L = [vy], \ \ M = [wy]
\]
and let $ \eta = 2\pi y $.
We deform the contour integral path $ C $ to the straight lines 
$ C_1, C_2, C_3 $ and $ C_4 $ joining 
$ \infty, c\eta + i \eta (1+c), -c \eta + i(1-c) \eta, -c \eta - (2L+1)\pi i, \infty $
where $ c $ is an absolute constant $ 0 < c \leq 1/2 $.

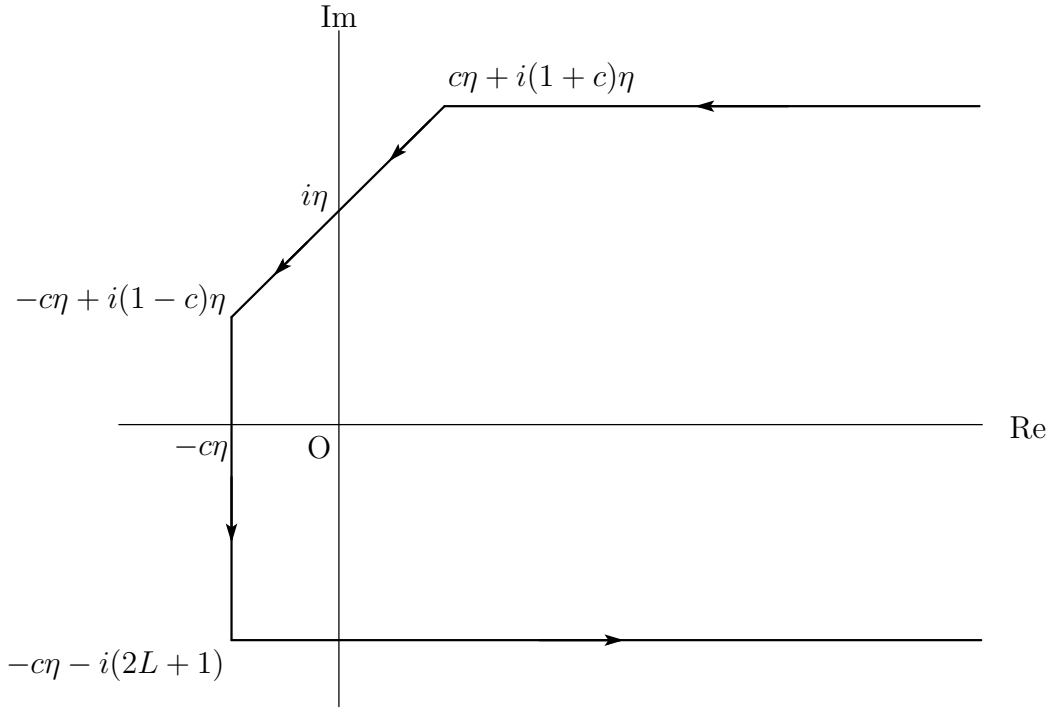
\begin{center}
{\unitlength 0.1in%
\begin{picture}( 50.8200, 36.8000)(  4.7000,-41.1000)%
%
\special{pn 13}%
\special{pa 5552 3762}%
\special{pa 1628 3762}%
\special{fp}%
\special{pa 1628 3762}%
\special{pa 1628 2070}%
\special{fp}%
\special{pa 1628 2070}%
\special{pa 2743 965}%
\special{fp}%
\special{pa 2743 965}%
\special{pa 5543 965}%
\special{fp}%
\put(16.0000,-20.6100){\makebox(0,0)[rb]{$ -c \eta + i(1-c)\eta $}}%
\put(21.3400,-15.1300){\makebox(0,0)[rb]{$ i\eta $}}%
\put(27.6100,-9.0000){\makebox(0,0)[lb]{$ c\eta + i(1+c)\eta $}}%
\put(57.9600,-26.4700){\makebox(0,0){$ \rm{Re} $}}%
\put(15.9100,-38.0800){\makebox(0,0)[rt]{$ -c\eta -i(2L+1) $}}%
\put(21.9000,-4.9500){\makebox(0,0){$ \rm{Im} $}}%
\put(20.8700,-27.4900){\makebox(0,0){O}}%
\put(14.6900,-27.5800){\makebox(0,0){$ -c\eta $}}%
%
\special{pn 13}%
\special{pa 4542 965}%
\special{pa 4074 965}%
\special{fp}%
\special{sh 1}%
\special{pa 4074 965}%
\special{pa 4141 985}%
\special{pa 4127 965}%
\special{pa 4141 945}%
\special{pa 4074 965}%
\special{fp}%
%
\special{pn 13}%
\special{pa 2668 1039}%
\special{pa 2466 1239}%
\special{fp}%
\special{sh 1}%
\special{pa 2466 1239}%
\special{pa 2527 1206}%
\special{pa 2504 1201}%
\special{pa 2499 1178}%
\special{pa 2466 1239}%
\special{fp}%
%
\special{pn 13}%
\special{pa 2024 1678}%
\special{pa 1863 1837}%
\special{fp}%
\special{sh 1}%
\special{pa 1863 1837}%
\special{pa 1924 1804}%
\special{pa 1901 1800}%
\special{pa 1896 1776}%
\special{pa 1863 1837}%
\special{fp}%
%
\special{pn 13}%
\special{pa 3239 3762}%
\special{pa 3661 3762}%
\special{fp}%
\special{sh 1}%
\special{pa 3661 3762}%
\special{pa 3594 3742}%
\special{pa 3608 3762}%
\special{pa 3594 3782}%
\special{pa 3661 3762}%
\special{fp}%
%
\special{pn 13}%
\special{pa 1628 2907}%
\special{pa 1628 3232}%
\special{fp}%
\special{sh 1}%
\special{pa 1628 3232}%
\special{pa 1648 3165}%
\special{pa 1628 3179}%
\special{pa 1608 3165}%
\special{pa 1628 3232}%
\special{fp}%
%
\special{pn 8}%
\special{pa 1038 2633}%
\special{pa 5552 2633}%
\special{fp}%
\special{pa 5552 2633}%
\special{pa 5552 2633}%
\special{fp}%
%
\special{pn 8}%
\special{pa 2190 570}%
\special{pa 2190 4110}%
\special{fp}%
\end{picture}}%

\end{center}

Next we consider the residue of the integrand of (\ref{contour_of_Barnes})
\[
	F(z) = \frac{z^{s-1}e^{-(\alpha+vN+wN)z}}{(e^{vz}-1)(e^{wz}-1)}.
\]
\begin{enumerate}[(i)]
	\item
	In the case when $ v, w $ are linear independent over $ \Q $, 
	$ F(z) $ has simple poles at	
	\[
		z = \frac{2\pi in}{v},\ \frac{2\pi in}{w} \ (n = \pm1,\ \pm2,\ \cdots).
	\]
	Since
	\begin{eqnarray*}
		\lim_{z \rightarrow 2\pi in/v}\left(z - \frac{2\pi in}{v}\right) F(z)
		&=& \lim_{z \rightarrow 2\pi in/v}\left(z - \frac{2\pi in}{v}\right)	
						\frac{z^{s-1}e^{-(\alpha + vN + wN)z}}{(e^{vz}-1)(e^{wz}-1)}	\\
		&=& \lim_{z \rightarrow 2\pi in/v} \left( \frac{e^{vz}-e^{2\pi in}}{z-2\pi in/v} \right)^{-1}
			\frac{z^{s-1} e^{-(\alpha + vN+wN)z}}{e^{wz}-1}	\\
		&=& \frac{1}{v} \left(\frac{2\pi in}{v} \right)^{s-1}
			\frac{e^{-(\alpha +wN)2\pi in/v}}{e^{2\pi inw/v}-1},
	\end{eqnarray*}
	then we have
	\begin{eqnarray*}
		\mathop{\mathrm{Res}}\limits_{z = 2\pi in/v} F(z) 
		&=& \frac{1}{v} \left(\frac{2\pi in}{v} \right)^{s-1}
			\frac{e^{-(\alpha +wN)2\pi in/v}}{e^{2\pi inw/v}-1}		\\
		&=&  \begin{cases}
			v^{-s} e^{-2\pi in \alpha /v} (2\pi n)^{s-1}e^{\pi i(s-1)/2}  & (n>0) \\
			v^{-s} e^{2\pi in \alpha /v} (-2\pi n)^{s-1}e^{3\pi i(s-1)/2} & (n<0)
		\end{cases}
	\end{eqnarray*}
	and we obtain
	\begin{eqnarray}
		&& \zeta_2 (s, \alpha ; v,w)	\nonumber \\
		&& = \mathop{\sum \sum} \limits_{0 \leq m,n \leq N} \frac{1}{(\alpha + vm +wn)^s}
		 + \frac{1}{w^s} \sum_{m=0}^N \zeta_H^*(s, \alpha_{v,m})
		 + \frac{1}{v^s} \sum_{n=0}^N \zeta_H^*(s, \alpha_{w,n})		\nonumber \\
		&& \ \ - \frac{\Gamma(1-s)}{(2\pi i)^{1-s}e^{\pi is}}
			\left\{ \frac{1}{v^s} \sum_{0<|n|\leq L} 
				\frac{e^{2\pi in(\alpha+wN)/v}}{(e^{2\pi inw/v}-1)n^{1-s}} 
				+ \frac{1}{w^s} \sum_{0<|n|\leq M} \frac{e^{-2\pi in(\alpha+vN)/w}}{(e^{2\pi inv/w}-1)n^{1-s}} \right\}
			\nonumber \\
		&& \ \ + \frac{1}{\Gamma(s)(e^{2\pi is}-1)} \left( \int_{C_1} + \int_{C_2} + \int_{C_3} + \int_{C_4} \right)
			\frac{z^{s-1}e^{-(\alpha + vN + wN)z}}{(e^{vz} -1)(e^{wz}-1)} dz. 	\label{zeta_2_int}
	\end{eqnarray}
	From here, we consider the order of the integral term on the right-hand side of (\ref{zeta_2_int}).
	First, we consider it on the integral path $ C_4 $. Let $ z = u +iu' = re^{i \theta } $ then 
	$ |z^{s-1}| = r^{\sigma -1} e^{-t \theta} $. Since
	$ \theta \geq (5/4)\pi, r \asymp u + c \eta, |e^{vz}-1|\gg 1 $ and $ |e^{wz}-1|\gg 1 $ we have
	\begin{eqnarray}
		\int_{C_4} F(z) dz
		&=& \int_{C_4} \frac{z^{s-1}e^{-(\alpha+vN+wN)z}}{(e^{vz}-1)(e^{wz}-1)}dz	\nonumber	\\
		&\ll& e^{-(5/4)\pi t} \int_{-c\eta}^\infty (u+c\eta)^{\sigma -1} e^{-(\alpha + vN + wN)u} du	\nonumber	\\
                &\ll& e^{(\alpha + vN + wN)c \eta -(5/4)\pi t} 
                	\int_0^\infty u^{\sigma -1} e^{-(\alpha + vN + wN)u} du		\nonumber	\\
		&\ll& e^{(\alpha + vN + wN)c \eta -(5/4)\pi t} (\alpha + vN + wN)^{-\sigma} \Gamma(\sigma)	\nonumber	\\
		&\ll& x^{-\sigma} e^{c \alpha \pi y + (c-(5/4)\pi)t} \label{int C_4}
	\end{eqnarray}
	Secondly, we consider the order of the integral on $ C_3 $ of (\ref{zeta_2_int}).
	Noticing
	\[
		\arctan{\varphi} = \int_0^\varphi \frac{d\mu}{1+\mu^2} > \int_0^\varphi \frac{d\mu}{(1+\mu)^2}
						= \frac{\varphi}{1+\varphi},
	\]
	at $ \varphi > 0 $, we have
	\[
		\theta = \arg{z} = \frac{\pi}{2} + \arctan{\frac{c}{1-c}} = \frac{\pi}{2} + c
	\]
	on $ C_3 $. Then we have
	\begin{eqnarray*}
		|z^{s-1}e^{-(\alpha + vN +wN)z}| 
		& \ll & \eta^{\sigma-1} e^{-(\pi/2 + c)t} e^{(\alpha + vN +wN)c\eta}	\\
		& \ll & \eta^{\sigma-1} e^{- \pi t/2}.
	\end{eqnarray*}
	Also, since $ | e^{vz} - 1| \gg 1 $, $ |e^{wz}-1| \gg 1 $ we have
	\begin{eqnarray}
		\int_{C_3} F(z) dz
		&=& \int_{C_3} \frac{z^{s-1}e^{-(\alpha+vN+wN)z}}{(e^{vz}-1)(e^{wz}-1)}dz	\nonumber \\
		& \ll & \int_{-(2L+1)}^{(1-c)\eta} \eta^{\sigma-1} e^{-\pi t/2} du'
		\ll \eta^{\sigma} e^{-\pi t/2}		\label{int C_3}				
	\end{eqnarray}
	Thirdly, since $ |e^{vz}-1| \gg e^{vu} $ and $ |e^{wz}-1| \gg e^{wu} $ on $ C_1 $, we have
	\begin{eqnarray*}
	&& \frac{z^{s-1}e^{-(\alpha+ vN + wN)z}}{(e^{vz}-1)(e^{wz}-1)} 	\\
	&& \qquad \qquad 
		\ll \eta^{\sigma -1} \exp{\left(-t\arctan{\frac{(1+c)\eta}{u}} - (\alpha + (N+1)(v+w))u \right)}.
	\end{eqnarray*}
	Since $ N + 1 \geq x/(v+w) = t/(v+w)\eta $ are included in the fractional part of the 
	right hand-side of the above $ -(\alpha + (N+1)(v+w))u $ may be replaced with $ tu/\eta $.
	Also, since
	\[
		\frac{d}{du} \left( \arctan{\frac{(1+c)\eta}{u}} + \frac{u}{\eta} \right)
		= - \frac{(1+c)\eta}{u^2 + (1+c)^2 \eta^2} + \frac{1}{\eta} > 0
	\]
	and
	\[
		\arctan{\varphi} = \int_0^\varphi \frac{d\mu}{1+\mu^2} < \int_0^\varphi d\mu = \varphi,
	\]
	we have
	\begin{eqnarray*}
		\arctan{\frac{(1+c)\eta}{u}} + \frac{u}{\eta}
		& \geq & \arctan{\frac{1+c}{c}} + c 
		 =  \frac{\pi}{2} - \arctan{\frac{c}{1+c}} + c	\\
		& > & \frac{\pi}{2} + A(c)
	\end{eqnarray*}
	in $ c\eta \leq u \leq \pi \eta $, and let $ A(c) = c^2/(1+c)^2 $. Then we have
	\[
		\frac{z^{s-1}e^{-(\alpha +vN+wN)z}}{(e^{vz}-1)(e^{wz}-1)}
		\ll \eta^{\sigma -1} \exp{\left( - \left( \frac{\pi}{2} + A(c) \right)t \right)}.
	\]
	Also, since
	\[
		\frac{z^{s-1}e^{-(\alpha +vN+wN)z}}{(e^{vz}-1)(e^{wz}-1)}
		\ll \eta^{\sigma -1} \exp{(-(\alpha + vx + wx)u)}
	\]
	in $ u \geq \pi \eta $, then we obtain
	\begin{eqnarray}
		&& \int_{C_1} \frac{z^{s-1}e^{-(\alpha +vN+wN)z}}{(e^{vz}-1)(e^{wz}-1)} dz			\nonumber \\
		&& \qquad \ll  \eta^{\sigma -1} \left\{ \int_{c\eta}^{\pi \eta} e^{- ( \pi/2 + A(c))t} du 
			+ \int_{\pi \eta}^\infty e^{-(\alpha + vx + wx)u} du \right\}					\nonumber \\
		&& \qquad \ll  \eta^\sigma e^{- ( \pi/2 + A(c))t} 
			+ \eta^{\sigma-1} e^{-(\alpha + vx + wx)\pi \eta }								\nonumber \\
		&& \qquad \ll \eta^\sigma e^{- ( \pi/2 + A(c))t}	\label{int C_1}.
	\end{eqnarray}
	Finally, we describe the integral evaluation on $ C_2 $. 
	Since, it can be rewritten that $ z = i\eta + \xi e^{\pi i/4} $
	(where $ \xi \in \R $ and $ |\xi| \leq \sqrt{2}c\eta $ ), we have 
	\begin{eqnarray*}
		z^{s-1} &=& \exp\left\{ (s-1) \left( \frac{\pi i}{2} 
				+ \log{(\eta + \xi e^{-\pi i/4})} \right) \right\}	\\
                	&=& \exp\left\{ (s-1) \left( \frac{\pi i}{2} 
				+ \log{\eta} + \log{\left(1 + \frac{\xi}{\eta} e^{-\pi i/4} \right)} \right) \right\}	\\
			&=& \exp\left\{ (s-1) \left( \frac{\pi i}{2} 
				+ \log{\eta} + \frac{\xi}{\eta}e^{-\pi i/4} 
				- \frac{\xi^2}{2\eta^2}e^{-\pi i/2} + O\left( \frac{\xi^3}{\eta^3}\right) 
				\right) \right\}	\\
			&\ll& \eta^{\sigma-1} 
				\exp\left\{ \left( -\frac{\pi}{2} + \frac{\xi}{\sqrt{2}\eta} 
						- \frac{\xi^2}{2\eta^2} + O\left( \frac{\xi^3}{\eta^3}\right) 
				\right)t \right\} 	\quad (\xi \rightarrow \infty).
	\end{eqnarray*}
	as $ \eta \rightarrow \infty $. Also, since
	\begin{eqnarray*}
		\frac{e^{-(\alpha + vN + wN)z}}{(e^{vz} - 1)(e^{wz} - 1)}
		&=& \frac{e^{-(\alpha + vN + wN)z + (\alpha + vx + wx)z}}{(e^{vz} - 1)(e^{wz} - 1)} \cdot 
			e^{-(\alpha + vx + wx)z}	\\
		&=& \frac{e^{(v+w)(x-N)z}}{(e^{vz} - 1)(e^{wz} - 1)} \cdot e^{-(\alpha + vx + wx)z}
	\end{eqnarray*}
	and
	\[
		\frac{e^{(v+w)(x-N)z}}{(e^{vz} - 1)(e^{wz} - 1)} \ll 
		\begin{cases}
			e^{(v+w)(x-N-1)u} & \left( u > \dfrac{\pi}{2} \right) \\
			e^{(v+w)(x-N)u} & \left( u < -\dfrac{\pi}{2} \right),
		\end{cases}
	\]	
	we have
	\[
		\frac{e^{-(\alpha + vN + wN)z}}{(e^{vz} - 1)(e^{wz} - 1)}
		\ll |e^{-(\alpha + vx + wx) \xi/ \sqrt{2}}| 	\quad \left( |u| > \frac{\pi}{2} \right).
	\]
	Hence
	\begin{eqnarray}
		&& \int_{C_2 \cap \{z \,|\, |u|>\pi/2\}} \frac{z^{s-1}e^{-(\alpha + vN + wN)z}}
			{(e^{vz} - 1)(e^{wz} - 1)} dz	 \nonumber\\
		&& \ll \eta^{\sigma -1} e^{-\pi t/2} \int_{-\sqrt{2}c\eta}^{\sqrt{2}c\eta} 	
			 \exp\left\{ \left( \frac{\xi}{\sqrt{2} \eta}(1-v-w) 
						- \frac{\xi^2}{2 \eta^2} + O\left( \frac{\xi^3}{\eta^3}\right) 
				\right)t \right\} d\xi		\nonumber\\
		&& \ll \eta^{\sigma -1} e^{- \pi t/2} \int_{-\infty}^\infty 
			\exp\left\{ - \frac{B(c)\xi^2 t}{\eta^2} \right\} d\xi	\nonumber\\
		&& \ll \eta^\sigma t^{-1/2} e^{-\pi t/2}, \label{evaluate C_2_1}
	\end{eqnarray}
	where $ B(c) $ is a constant depending on $ c $.
	The argument can also be applied to the part $ |u| \leq \pi/2 $ if $ |e^{z-2\pi i \lambda}| > A $. 
	If not, that is the case when the contour goes too near to the pole at $ z = 2L\pi i/v $ (or $ 2M\pi i/w $),
	we take an arc of the circle $ |z - 2L\pi i/v| = \varepsilon $ (or $ |z - 2M\pi i/w| = \varepsilon $) .
	On this arc we can write
	\[
		z = \frac{2L\pi i}{v} + \varepsilon e^{i \beta} \quad \mathrm{or} \quad 
		z = \frac{2M\pi i}{w} + \varepsilon e^{i \beta},
	\]
	where $ \varepsilon $ is a positive real number less than the distance between any two poles, that is,
	\[
		0 < \varepsilon < \min_{k, l} \left\{ \left| \frac{2k\pi i}{v} - \frac{2l\pi i}{w} \right| 
					\, \left| \ 0 < \frac{2k\pi}{v}, \frac{2l\pi}{w}  < \eta \right. \right\}.
	\]
	Then,
	\begin{eqnarray*}
		\log{(z^{s-1})} 
		&=& (s-1) \log{\left( \frac{2L\pi i}{v}  + \varepsilon e^{i \beta} \right)}	\\
		&=& (s-1) \log{e^{\pi i/2} \left( \frac{2L\pi}{v} + \frac{\varepsilon e^{i \beta}}{i} \right)}	\\
		&=& (\sigma + it - 1) \left\{ \frac{\pi i}{2} + \log{\frac{2L\pi}{v}} 
			+ \log\left( 1 + \frac{v \varepsilon e^{i \beta}}{2L\pi i} \right) \right\}	\\
		&=& -\frac{\pi t}{2} + (s-1)\log{\frac{2L\pi}{v}} + \frac{v \varepsilon e^{i \beta}}{2L\pi}t + O(1).
	\end{eqnarray*}
	On the last line of the above calculations, we used $ N^2 \gg t $ which follows from the assumption $ x \leq y $.
	Then
	\begin{eqnarray*}
		&& z^{s-1} e^{-(\alpha + vN + wN)z}	\\
		&& = \exp{\left( -\frac{\pi t}{2} + (s-1)\log{\frac{2L\pi}{v}} + \frac{v \varepsilon e^{i \beta}}{2L\pi } t+ O(1) \right)}	\\
		&& 	\qquad \qquad \qquad \qquad 
			\times \exp{\left( -(\alpha + vN +wN) \left( \frac{2L\pi i}{v} + \varepsilon e^{i \beta} \right) \right)}		\\
		&& = \exp{\left( -\frac{\pi t}{2} + (s-1)\log{\frac{2L\pi}{v}} 
			+ \left( \frac{vt}{2\pi L} - (\alpha + vN + wN) \right) \varepsilon e^{i \beta} + O(1) \right)},
	\end{eqnarray*}
	and since
	\begin{eqnarray*}
		&& \left( \frac{vt}{2\pi L} - (\alpha + vN + wN) \right) \varepsilon e^{i \beta}	\\
		&& \qquad = \frac{vt - (\alpha + vN + wN)2\pi L}{2\pi L} \varepsilon e^{i \beta}		\\
		&& \qquad = \frac{2\pi xyv - 2\alpha \pi L - (v+w)[x/(v+w)]2\pi L}{2\pi L} \varepsilon e^{i \beta}	\\
		&& \qquad = - \alpha \varepsilon e^{i \beta} + \frac{2\pi xyv - (v+w)[x/(v+w)]2\pi L}{2\pi L} \varepsilon e^{i \beta}	\\
		&& \qquad \asymp 
                	 \left\{ x - (v+w) \left[ \frac{x}{v+w} \right] - \alpha \right\}
                	 \varepsilon e^{i \beta}
			= O(1)
	\end{eqnarray*} 
	we have
	\begin{eqnarray*}
		z^{s-1} e^{-(\alpha + vN + wN)z}	
		&=& \exp{\left( -\frac{\pi t}{2} + (s-1)\log{\frac{2L\pi}{v}} + O(1) \right)}	\\
		&\ll & \left( \frac{L}{v} \right)^{\sigma - 1} e^{-\pi t/2} = O(\eta^{\sigma - 1}e^{-\pi t/2})
	\end{eqnarray*}
	In the case when the path $ C_2 $ is running around the another poles at $ z = 2k \pi i/w + \varepsilon e^{i \beta} $, 
        by using the similar above method to obtain
	\[
		z^{s-1} e^{-(\alpha + vN + wN)z} = O(\eta^{\sigma - 1}e^{-\pi t/2}).
	\]
	Therefore together with (\ref{evaluate C_2_1}), we have
	\begin{equation}
		 \int_{C_2} \frac{z^{s-1}e^{-(\alpha + vN + wN)z}}{(e^{vz} - 1)(e^{wz} - 1)} dz
		\ll \eta^\sigma t^{-1/2} e^{-\pi t/2} + \eta^{\sigma - 1}e^{-\pi t/2} \label{int C_2}
	\end{equation}
	Since, the evaluation of all integrals was obtained, using the evaluation formulas
	(\ref{int C_4}),(\ref{int C_3}),(\ref{int C_1}),(\ref{int C_2}) and 
	$ \Gamma(1-s) \ll t^{1/2 - \sigma} e^{\pi t/2} $, we find that the evaluation of the integral term of 
	(\ref{zeta_2_int}) is
	\begin{eqnarray*}
		&\ll& t^{1/2 - \sigma} e^{\pi t/2} 
			\{ \eta^\sigma e^{-(\pi/2 + A(c))t}
			+ \eta^\sigma t^{-1/2} e^{-\pi t/2} + \eta^{\sigma-1}e^{-\pi t/2} 	\\	
		&& \qquad \qquad \qquad \qquad \qquad \qquad \qquad \qquad
			+ \eta^\sigma e^{-\pi t/2 } + x^{-\sigma} e^{(c - 5\pi/4)t} \}  \\
		&\ll& t^{1/2} \left( \frac{\eta}{t}\right)^\sigma e^{-A(c)t} 
			+ \left(\frac{\eta}{t}\right)^\sigma + t^{-1/2} \left(\frac{\eta}{t}\right)^{\sigma-1}
			+ t^{1/2 - \sigma} x^{-\sigma} e^{(c - 5\pi/4)t}	\\
		&\ll& e^{-\delta t} + x^{-\sigma} + t^{-1/2} x^{1-\sigma}  \ll  x^{-\sigma},
	\end{eqnarray*}
	where $ \delta $ is a small positive real number.
	Therefore we obtain (\ref{Thm1(i)}).
	
	\item
	In the case when $ v, w $ are linear dependent over $ \Q $, that is, there exist $ p, q \in \N $
	such that $ pv = qw $ and $ (p, q) = 1 $. $ F(z) $ has simple poles at	
	\[
		z = \frac{2\pi in}{v} \ \, (n \in \Z \setminus \{ 0 \},\ q \mid \hspace{-.65em}/  n) ,\ \
			\frac{2\pi in}{w} \ \, (n \in \Z \setminus \{ 0 \},\ p \mid \hspace{-.65em}/  n) .
	\]	
	On the other hand here for $ w = pv/q $,
        \begin{eqnarray*}
		&& \lim_{z \rightarrow 2q\pi in/v} \frac{d}{dz} 
				\left\{ \left( z -\frac{2q\pi in}{v} \right)^2 
						\frac{z^{s-1}e^{-(\alpha + vN+pvN/q)z}}{(e^{vz}-1)(e^{pvz/q}-1)} \right\}	\\
		&& = \lim_{z_0 \rightarrow 0} \frac{d}{dz_0} 
			\left\{ z_0^2
					\left(z_0 + \frac{2q\pi in}{v}\right)^{s-1} 
							\frac{e^{-(\alpha +vN + pvN/q)z_0}e^{-2q\pi in\alpha/v}}
					{(e^{vz_0}-1)(e^{pvz_0/q}-1)}	\right\}	\\
		&& = \lim_{z_0 \rightarrow 0} \frac{d}{dz_0} 
			\left\{ \vbox to 30pt{}  z_0^2
					\left(z_0 + \frac{2q\pi in}{v}\right)^{s-1}		\right.	\\
		&&  \ \  \qquad \qquad \qquad \qquad \left.  \times \frac{e^{-(\alpha +vN + pvN/q)z_0}e^{-2q\pi in\alpha/v}}
					{\left( vz_0 + \frac{v^2}{2!}z_0^2 + O(z_0^3) \right)
					 \left( \frac{pv}{q}z_0 + \frac{1}{2!}\left(\frac{pv}{q}\right)^2 z_0^2 + O(z_0^3) \right)}
					\right\}	\\
		&& = \frac{q}{pv^2} (s-1) \left( \frac{2q\pi in}v \right)^{s-2} e^{-2q\pi in\alpha /v} \nonumber \\
        && \ \  \qquad \quad  - \left( \frac{\alpha q}{pv^2} + \frac{(p+q)N}{pv} + \frac{vp}{2q} + \frac{v}2 \right) 
                	\left( \frac{2q\pi in}v \right)^{s-1} e^{-2q\pi in\alpha /v}
	\end{eqnarray*}
        Therefore $ F(z) $ has double poles at
	\[
		z = \frac{2\pi i qn}{v} = \frac{2\pi ipn}{w} \quad (n \in \Z \setminus \{ 0 \}).
	\]
	Then, we calculate the following residue sum;
	\begin{eqnarray*}
		&& \sum_{0 < |n| \leq M} \mathrm{Res} F(z)
			= \mathop{\sum_{0 < |n| \leq L}}\limits_{q \; \mid \hspace{-.25em}/ \, n}
				\mathop{\mathrm{Res}}\limits_{z = 2\pi i n/v} F(z)
			+ \mathop{\sum_{0 < |n| \leq K}}\limits_{p \; \mid \hspace{-.25em}/ \, n}
				\mathop{\mathrm{Res}}\limits_{z = 2\pi i n/w} F(z)	\\
		&& \qquad \qquad \qquad \qquad 
			+ \sum_{0 < |k| \leq K} \mathop{\mathrm{Res}}\limits_{z = 2\pi i qk/v} F(z)
	\end{eqnarray*}
	
	therefore, we have
	\begin{eqnarray*}
		\mathop{\mathrm{Res}}\limits_{z = 2\pi in/v} F(z) 
		&=& \frac{1}{v} \left(\frac{2\pi in}{v} \right)^{s-1}
			\frac{e^{-(\alpha +wN)2\pi in/v}}{e^{2\pi inw/v}-1}		\\
		&=&  \begin{cases}
			e^{-2\pi in \alpha} (2\pi n)^{s-1}e^{\pi(s-1)/2}  & (n>0) \\
			e^{2\pi in \alpha} (-2\pi n)^{s-1}e^{3\pi(s-1)/2} & (n<0)
		\end{cases}
	\end{eqnarray*}
	and we obtain
	\begin{eqnarray*}
		&& \zeta_2(s, \alpha; v,w)		\nonumber \\
			&&  = \mathop{\sum\sum}\limits_{0 \leq m,n \leq N} \frac{1}{(\alpha + vm + wn)^s}
				+ \frac{1}{w^s}\sum_{m=0}^N \zeta_H^*(s,\alpha_{v,m}) 
						+ \frac{1}{v^s}\sum_{n=0}^N \zeta_H^*(s, \alpha_{w,n})		\nonumber \\
			&& \ \ - \frac{\Gamma(1-s)}{(2\pi i)^{1-s}e^{\pi is}}\left\{ \frac{1}{v^s} 
					\mathop{\sum_{0 < |n| < L}}\limits_{q \; \mid \hspace{-.25em}/ \, n} 
							\frac{e^{-2\pi in(\alpha+wN)/v}}
								{(e^{2\pi inw/v}-1)n^{1-s}} + \frac{1}{w^s}
					\mathop{\sum_{0 < |n| < M}}\limits_{p \; \mid \hspace{-.25em}/ \, n} 
							\frac{e^{-2\pi in(\alpha+vN)/w}}
								{(e^{2\pi inv/w}-1)n^{1-s}} \right. \nonumber \\
		&& \qquad
                	+ \frac{q^{s-1}}{2\pi ipv^s}(1-s) \sum_{0 < |n| < M} \frac{e^{-2q\pi in\alpha/v}}{n^{2-s}} 
                        \nonumber \\
                && \qquad	
                	- \left. \left( \frac{\alpha q}{pv^2} + \frac{(p+q)N}{pv} + \frac{vp}{2q} + \frac{v}2 \right) 
                	\left( \frac{q}{v} \right)^{s-1} \sum_{0 < |n| < M} \frac{e^{-2q\pi in\alpha /v}}{n^{1-s}} 
                        	 	\vbox to 35pt{} \right\} \nonumber \\
		&& \ \	+ \frac{1}{\Gamma(s)(e^{2\pi is}-1)} \left( \int_{C_1} + \int_{C_2} + \int_{C_3} + \int_{C_4} \right)
			\frac{z^{s-1}e^{-(\alpha + vN + wN)z}}{(e^{vz} -1)(e^{wz}-1)} dz. 	
	\end{eqnarray*}
	Furthermore, four integrals in the last term of the above are evaluated to lead the same result by the
	similar method as in (i)
\end{enumerate}
Hence proof of Theorem \ref{th:Main_Theorem1} is complete.
\qed

\bigskip

\section{Some Lemmas}
In this section, we will first introduce two basic lemmas (Lemma \ref{f''} and Lemma \ref{f'''})
on exponential sum in the van der Corput method. These lemmas are called 
the second order differential test and the third order differential test, respectively.
Also, the Hurwitz zeta-function $ \zeta_H(s,\alpha) $ has similar results to the Hardy-Littlewood's results
(\ref{AFE}) and (\ref{Hardy-Littlewood}), which are described in Lemma \ref{AFE_of_zeta_H}
and Lemma \ref{Hardy-Littlewood_zeta_H}.
There are approximate functional equation for the Lerch zeta-function $ \zeta_L(s, \alpha, \lambda) $ proved by
R. Garunk\v{s}tis, A. Laurin\v{c}ikas, and J. Steuding (see \cite{GLS}, \cite{LG0}, \cite{LG}),
Lemma \ref{AFE_of_zeta_H} is an analogue of its special case (see \cite{TM1}) .

\bigskip

\begin{lemma}[Theorem 5.9 in \cite{Tit}]\label{f''}
Let $ a,b $ are real number with $ b \geq a+1 $ and $ c > 1 $ is a constant.
Suppose that $ f(x) $ be a real two times differentiable function which satisfies
\[
	0 < \Lambda \leq f''(x) \leq c \Lambda
        \quad or \quad
        0 < \Lambda \leq -f''(x) \leq c \Lambda
\]
in $ [a,b] $. Then,
\[
	\sum_{a < n \leq b} e^{2\pi if(n)}
        \ll c (b-a) \Lambda^{1/2} + \Lambda^{-1/2}.
\]
\end{lemma}

\bigskip

\begin{lemma}[Theorem 5.11 in \cite{Tit}]\label{f'''}
Let $ a,b $ are real number with $ b \geq a+1 $, and $ c > 1 $ is a constant.
Suppose that $ f(x) $ be a real three times differentiable function which satisfies
\[
	0 < \Lambda \leq f'''(x) \leq c \Lambda
        \quad or \quad
        0 < \Lambda \leq -f'''(x) \leq c \Lambda
\]
in $ [a,b] $. Then,
\[
	\sum_{a < n \leq b} e^{2\pi if(n)}
        \ll c^{1/2} (b-a) \Lambda^{1/6} + (b-a)^{1/2}\Lambda^{-1/6}.
\]
\end{lemma}

\bigskip


\begin{lemma}[Theorem 1(ii) in \cite{TM1}]\label{AFE_of_zeta_H}
Let $ 0 < \alpha \leq 1 $. Suppose that $ 0 \leq \sigma \leq 1,\ x \geq 1,\ y\geq 1 $ and $ 2\pi xy = |t| $. Then
\begin{eqnarray}
&& \zeta_H(s, \alpha) = \sum_{0 \leq n \leq x}\frac{1}{(n+\alpha)^s} 	\nonumber \\
&& \qquad \qquad \qquad
	+ \frac{\Gamma(1-s)}{(2\pi)^{1-s}} 
	\left\{ e^{\frac{\pi i}{2}(1-s)} 
	\sum_{1 \leq n \leq y} \frac{e^{2\pi in(1-\alpha)}}{n^{1-s}} + 
	e^{-\frac{\pi i}{2}(1-s)} 
        \sum_{1 \leq n \leq y} \frac{e^{2\pi in \alpha}}{n^{1-s}} \right\} \nonumber \\
&& \qquad \qquad  \quad + O(x^{-\sigma}) + O(|t|^{1/2-\sigma}y^{\sigma-1}). 
	\label{AFE_of_Hurwitz}
\end{eqnarray} 
\end{lemma}

\bigskip

\begin{lemma}\label{Hardy-Littlewood_zeta_H}
Let $ x>0, t>0 $ and $ 0 < \alpha \leq 1 $. Suppose that $ x \ll t $, then
\[
	\sum_{0 \leq n \leq x} \frac1{(n+\alpha)^{1/2 + it}} \ll t^{1/6} \log{x}
        \qquad \quad (t \rightarrow \infty).
\]
\end{lemma}
\begin{proof}
First, we consider the following single series and by using Lemma \ref{f'''}
with $ f(x) = -t (2\pi)^{-1} \log{(x + \alpha)} $ . Suppose that $ a+1 \leq b \leq 2a $, then
\begin{eqnarray}
	\sum_{a < n \leq b} (n + \alpha)^{-it}
	= \sum_{a < n \leq b} e^{2\pi i f(n)} 
	&\ll & a \left\{ \frac{t}{(a+\alpha)^3} \right\}^{1/6} + a^{1/2} \left\{ \frac{t}{(a+\alpha)^3}\right\}^{-1/6}	\nonumber \\
	&\ll & a^{1/2} t^{1/6} + a t^{-1/6}.
		\label{a,t}
\end{eqnarray}
Also by using partial summation formula, and by using (\ref{a,t}), then
\begin{eqnarray*}
	&& \sum_{a < n \leq b} (n+ \alpha)^{-1/2-it}  	\\
	&& \qquad = \left\{ \sum_{a < n \leq b} (n+ \alpha)^{-it} \right\} (b-a)^{-1/2} 
			+ \frac12 \int_a^b \left\{ \sum_{a < n \leq \xi} (n+ \alpha)^{-it} \right\} \xi^{-3/2} d \xi	\\
	&& \qquad \ll (a^{1/2} t^{1/6} + a t^{-1/6} ) a^{-1/2}  
			+ \int_a^b ( a^{1/2} t^{1/6} + a t^{-1/6} ) \xi^{-3/2} d \xi	\\
	&& \qquad \ll  t^{1/6} + a^{1/2}t^{-1/6}.
\end{eqnarray*}
If $ a \ll t^{2/3} $, the above evaluation formula is
\begin{equation}
	\sum_{a < n \leq b} (n+ \alpha)^{-1/2-it} \ll  t^{1/6}.
	\label{a<<t^2/3}
\end{equation}
Also if $ t^{2/3} \ll a \ll t $, by using Lemma \ref{f''} with same $ f(x) $ as above method,
then we have
\begin{eqnarray*}
	\sum_{a < n \leq b} (n + \alpha)^{-it}
	&\ll & a \left\{ \frac{t}{(a+\alpha)^2} \right\}^{1/2} + a^{1/2} \left\{ \frac{t}{(a+\alpha)^2}\right\}^{-1/2}	\nonumber \\
	&\ll & t^{3/2} + a t^{-1/2}.
\end{eqnarray*}
Similarly, by using partial summation formula to obtain
\[
	\sum_{a < n \leq b} (n + \alpha)^{-1/2-it} \ll a^{-1/2} t^{1/2} + a^{1/2} t^{-1/2}.
\]
Therefore, also in the case of $ t^{3/2} \ll a \ll t $, the evaluation formula (\ref{a<<t^2/3}) holds.
Furthermore, setting $ (a,b) = (2^{-j}x, 2^{-j+1}x) $ and taking the sum for $ j = 1, 2, 3, \ldots $,
so calculate the sum of $ O(\log{x}) $ terms, then we obtain
\[
	\sum_{0 \leq n \leq x} (n + \alpha)^{-1/2-it} = \sum_{j \geq 1} \sum_{2^{-j}x < n \leq 2^{-j+1}x} (n+\alpha)^{-1/2-it}
	\ll t^{1/6} \log{x}.
\]
\end{proof}

\section{Proof of Theorem \ref{th:Main_Theorem2} and Theorem \ref{th:Main_Theorem3}}

\textbf{Proof of Theorem \ref{th:Main_Theorem2}}.
\begin{enumerate}
\item[(i)] In the case when $ v, w $ are linear independent over $ \Q $.
Setting $ s = 1/2 + it $ on Theorem \ref{th:Main_Theorem1} (i), and note that
\[
	- \frac{\Gamma(1-s)}{(2\pi i)^{1-s}e^{\pi is}} \sim 1
\]
so we have following evaluation,
\begin{eqnarray}
	&& \zeta_2 \left( \frac12+it, \alpha; v,w \right)		\nonumber \\
	&& \ll \mathop{\sum\sum} \limits_{0 \leq m,n \leq x/(v+w)} \frac1{(\alpha+vm+wn)^{1/2+it}} 
			+ \frac1{\sqrt{w}} \sum_{0 \leq m \leq x/(v+w)} \zeta_H^*\left( \frac12 +it, \alpha_{v,m} \right)	\nonumber \\
	&& \qquad \qquad + \frac1{\sqrt{v}} \sum_{0 \leq n \leq x/(v+w)} \zeta_H^*\left( \frac12 +it, \alpha_{w,n} \right)	\nonumber \\
	&& \quad + \frac{2\pi}{\sqrt{v}} \sum_{0 < |n| \leq vy} \frac{e^{-2\pi in(\alpha+wN)/v}}{e^{2\pi inw/v} -1} \cdot \frac1{n^{1/2-it}}
		+ \frac{2\pi}{\sqrt{w}} \sum_{0 < |n| \leq wy} \frac{e^{-2\pi in(\alpha+vN)/w}}{e^{2\pi inv/w} -1} \cdot \frac1{n^{1/2-it}}.
		\qquad
		\label{s=1/2+it}
\end{eqnarray}
We denote the right-hand side by $ S + T_1 + T_2 + U_1 + U_2 $.
To evaluation for double series $ S $  on (\ref{s=1/2+it}), we consider the following single series and by using Lemma \ref{f'''}
with $ g(x) = -t (2\pi)^{-1} \log{(\alpha +vm+wx)} $ and $ (a,b) = (2^{-j}x/(v+w),\ 2^{-j+1}x/(v+w)) $, we have
\begin{eqnarray}
	&& \sum_{0 \leq n \leq x/(v+w)} (\alpha+vm+wn)^{-it}
			= \sum_{j=1}^\infty \sum_{a < n \leq b} e^{2\pi i g(n)} \nonumber \\
	&& \qquad \quad \ll \sum_{j=1}^\infty \{ 2^{-j} x (m+x)^{-1/2} |t|^{1/6} + 2^{-j/2}x^{1/2}(m+x)^{1/2}|t|^{-1/6} \}	\nonumber \\
	&& \qquad \quad \ll x (m+x)^{-1/2} |t|^{1/6} + x^{1/2}(m+x)^{1/2}|t|^{-1/6}.
		\label{x,m,t}
\end{eqnarray}
Also by using partial summation formula, and by using (\ref{x,m,t}), then we have
\begin{eqnarray*}
	&& \sum_{0 \leq n \leq x/(v+w)} (\alpha+vm+wn)^{-1/2-it}	\\
	&& \qquad = \left\{ \sum_{0 \leq n \leq x/(v+w)} (\alpha+vm+wn)^{-it} \right\} \left( \frac{x}{v+w} \right)^{-1/2} \\
	&& \qquad \qquad 
			- \frac12(v+w)^{1/2} 
			\int_1^{x/(v+w)} \left\{ \sum_{0 \leq n \leq \xi/(v+w)} (\alpha+vm+wn)^{-it} \right\} \xi^{-3/2} d \xi	\\
	&& \qquad \ll x^{1/2} (m+x)^{-1/2} |t|^{1/6} + (m+x)^{1/2} |t|^{-1/6}  \\
	&& \qquad \qquad 
			- \int_1^{x/(v+w)} \left\{ \xi^{-1/2}(m+\xi)^{-1/2} |t|^{1/6} + \xi^{-1} (m+\xi)^{1/2} |t|^{-1/6} \right\} d \xi	\\
	&& \qquad \ll  x^{1/2} (m+x)^{-1/2} |t|^{1/6} + |t|^{-1/6} (m+x)^{1/2}  + |t|^{1/6} \log{(\sqrt{x}+ \sqrt{m+x})} 	\\
	&& \qquad \qquad + |t|^{-1/6} \sqrt{m} 
				\left\{ \log{\left( \sqrt{1+\frac{x}{m}} -1 \right)} - \log{\left( \sqrt{1+\frac{x}{m}}+1 \right)} \right\}.
\end{eqnarray*}
Furthermore by calculating the sum on $ m $, so we can evaluated to $ S $ as follows,
\begin{eqnarray}
	S
	& \ll & x^{1/2} |t|^{1/6} \sum_{0 \leq m \leq x} \frac1{\sqrt{m+x}} + |t|^{-1/6} \sum_{0 \leq m \leq x}\sqrt{m+x}  \nonumber \\
	&&+ |t|^{1/6} \sum_{0 \leq m \leq x} \log{(\sqrt{x}+ \sqrt{m+x})} 
						+ |t|^{-1/6} \sum_{0 \leq m \leq x} \sqrt{m} \log{ \left(1+ \frac{2m}{x} \right)} 		\nonumber \\
	& \ll & x |t|^{1/6} + x^{3/2} |t|^{-1/6} + |t|^{1/6} x \log{x} + |t|^{-1/6} x^{3/2}	\nonumber \\
	& \ll & |t|^{1/6} x \log{x} + |t|^{-1/6} x^{3/2}. 	\label{evaluate_S}
\end{eqnarray}
Next, we consider the order of $ T_1 + T_2 $. By using Lemma \ref{AFE_of_zeta_H}, we have
\begin{eqnarray*}
	&& \zeta_H^* \left( \sigma +it, \alpha_{v,m} \right)		\\
	&& \quad =  \zeta_H \left( \sigma +it, \alpha_{v,m} \right) - \sum_{0 \leq n \leq N+n_{v,m}} \frac{1}{(n+\alpha_{v,m})^{\sigma +it}}	\\
	&& \quad \ll  \zeta_H \left( \sigma +it, \alpha_{v,m} \right) - \sum_{0 \leq n \leq x} \frac{1}{(n+\alpha_{v,m})^{\sigma +it}}	\\
	&& \quad = \frac{\Gamma(1-s)}{(2\pi)^{1-s}} 
		\left\{ e^{\pi i(1-s)/2} \sum_{1 \leq n \leq y} \frac{e^{2\pi in(1-\alpha_{v,m})}}{n^{1-s}} 
			+ e^{-\pi i(1-s)} \sum_{1 \leq n \leq y} \frac{e^{2\pi in \alpha_{v,m}}}{n^{1-s}} \right\}		\\
	&& \quad \qquad  + O(x^{-\sigma}) + O(|t|^{1/2-\sigma}y^{\sigma-1})	\\
	&& \quad \ll |t|^{1/2-\sigma} e^{\pi(|t|-t)/2} \left| \sum_{1 \leq n \leq y} \frac{e^{2\pi in(1-\alpha_{v,m})}}{n^{1-s}} \right|
				+ |t|^{1/2-\sigma} e^{\pi(|t|+t)/2} \left| \sum_{1 \leq n \leq y} \frac{e^{2\pi in \alpha_{v,m}}}{n^{1-s}} \right|		\\
	&& \quad \qquad + x^{-\sigma} + |t|^{1/2 - \sigma} y^{\sigma-1}.
\end{eqnarray*}
Suppose that $ \sigma = 1/2 $ and by using Lemma \ref{Hardy-Littlewood_zeta_H}
\[
	\zeta_H^* \left( \frac12 +it, \alpha_{v,m} \right)
	\ll \left| \sum_{1 \leq n \leq y} \frac{e^{2\pi in \alpha_{v,m}}}{n^{1/2 - it}} \right|
	\ll \begin{cases}
		1						& (0 < \alpha_{v,m} < 1),	\\
		|t|^{1/6} \log{y}	& (\alpha_{v,m} = 1).
		\end{cases}
\]
Here, $ \alpha_{v,m} = 1 $ is equivalent to $ (\alpha+vm)/w \in \N $. 
We suppose that exists $ k \in \N $ such as $ (\alpha+vm)/w = k $ that is $ \alpha = kw - mv $.
Also taking $ k', m' $ such that $ \alpha = k'w - m'v $, so $ (k-k')w - (m'-m)v = 0 $ holds.
Since $v, w$ are linearly independent over $ \Q $, so $ (k', m')=(k, m) $. 
That is, there are at most one pair $ (k, m) $ that satisfies $ \alpha = kw-mv $.
Therefore, we have
\begin{eqnarray}
	T_1 + T_2 
	& \ll & \sum_{0 < m \leq x/(v+w)} \zeta_H^* \left( \frac12 +it, \alpha_{v,m} \right)  \nonumber \\
	& \ll & 
	\begin{cases}
		x & (\alpha, v, w \ \mathrm{are\ lin. \ indep. \ over}\ \Q),	\\
		x + |t|^{1/6} \log{y} & (\alpha, v, w\ \mathrm{are\ lin. \ dep.\ over}\ \Q).
	\end{cases}
	\label{evaluate_T}
\end{eqnarray}
Finally, we consider the evaluation of $ U_1 $ and $ U_2 $.
Let $ d(z, \varepsilon) $ be the closed disk whose center is $ z \in \C $ with radius $ \varepsilon > 0 $.
We first note that for any $ \varepsilon_1 > 0 $, then exists $ \delta = \delta(\varepsilon_1) > 0 $ 
such that for $ z \in \bigcup_{n \in \Z} d(2\pi in, \varepsilon_1) $ the inequality
\[
	\left| \frac1{e^z - 1} \right| \leq \delta e^{- \max\{ \mathrm{Re}(z),\ 0\}}
\]
holes.
Also, since $ v,w $ are linearly independent over $ \Q $ thus $ (\alpha + wN)/v \notin \Z $ and 
$ (\alpha + vN)/w \notin \Z $ for any $ N \in \N $, then we have
\begin{eqnarray*}
	&& \left| \sum_{0 < n \leq L} \frac{e^{-2\pi in(\alpha+wN)/v}}{e^{2\pi inw/v} -1} \cdot \frac1{n^{1/2-it}} \right|
	\leq \delta \left| \sum_{n = 1}^\infty \frac{e^{-2\pi in(\alpha+wN)/v}}{n^{1/2 - it}} \right|,	\\
	&& \left| \sum_{0 < n \leq M} \frac{e^{-2\pi in(\alpha+vN)/w}}{e^{2\pi inv/w} -1} \cdot \frac1{n^{1/2-it}} \right|
	\leq \delta \left| \sum_{n = 1}^\infty \frac{e^{-2\pi in(\alpha+vN)/w}}{n^{1/2 - it}} \right| 
\end{eqnarray*}
and each right-hand side series is convergent, so we have
\begin{equation}
	U_1 + U_2 \ll 1.		\label{evaluate_U}
\end{equation}
Since (\ref{evaluate_S}), (\ref{evaluate_T}) and (\ref{evaluate_U}), then we have
\begin{eqnarray*}
	&& \zeta_2 \left( \frac12 + it, \alpha; v,w \right) \ll |t|^{1/6} x \log{x} + |t|^{-1/6} x^{3/2}   \nonumber \\
	&& \qquad \qquad \qquad +
		\begin{cases}
				x & (\alpha, v, w \ \mathrm{are\ lin. \ indep. \ over}\ \Q),	\\
				x + |t|^{1/6} \log{y} & (\alpha, v, w\ \mathrm{are\ lin. \ dep.\ over}\ \Q).
		\end{cases}
\end{eqnarray*}
We consider in the case $ \alpha, v, w $ are linearly independent over $ \Q $. 
Let $ C $ is constant with $ C>1 $, taking $ x = C, y = |t|/2 \pi C $ then 
$ \zeta_2 \left(1/2 + it, \alpha; v,w \right) \ll |t|^{1/6} $.
On the other hand, in the case when $ \alpha, v, w $ are linearly dependent over $ \Q $, 
taking $ x = 2\pi (\log{t})^{1/2}, y = t (\log{t})^{-1/2} $ then,
$ \zeta_2 \left(1/2 + it, \alpha; v,w \right) \ll |t|^{1/6} \log{|t|} $
so we obtain the result of Theorem \ref{th:Main_Theorem2} (i).

\item[(ii)] In the case when $ v, w $ are linear dependent over $ \Q $.
Setting $ s = 1/2 + it $ on Theorem \ref{th:Main_Theorem1} (ii), we have
\begin{eqnarray}
		&& \zeta_2 \left(\frac12 + it, \alpha; v,w \right)		\nonumber \\
		&&  = \mathop{\sum\sum}\limits_{0 \leq m,n \leq N} 
			\frac{1}{(\alpha + vm + wn)^{1/2 + it}}
			+ \frac{1}{w^{1/2+it}}\sum_{m=0}^N \zeta_H^*\left(\frac12 + it,\alpha_{v,m} \right) 			\nonumber \\
		&& \qquad + \frac{1}{v^{1/2+it}}\sum_{n=0}^N \zeta_H^*\left( \frac12 + it, \alpha_{w,n} \right)		\nonumber \\
		&& \qquad -  \frac{\Gamma(1/2 - it)}{(2\pi i)^{1/2-it}e^{\pi i(1/2+it)}} \left\{
			\frac{2\pi}{v^{1/2+it}} 
			\mathop{\sum_{0 < |n| < L}}\limits_{q \; \mid \hspace{-.25em}/ \, n} 
			\frac{e^{-2\pi in(\alpha+wN)/v}}{(e^{2\pi inw/v}-1)n^{-1/2+it}} \right. 		\nonumber \\
		&& \qquad \qquad \qquad \qquad \qquad \qquad \qquad \left. + \frac{2\pi}{w^{1/2+it}}
			\mathop{\sum_{0 < |n| < M}}\limits_{p \; \mid \hspace{-.25em}/ \, n} 
			\frac{e^{-2\pi in(\alpha+vN)/w}}{(e^{2\pi inv/w}-1)n^{-1/2+it}}  
			\vbox to 35pt{} \right\}  \nonumber \\
		&& \qquad
            - \frac{\Gamma(1/2 - it)}{(2\pi i)^{1/2-it}e^{\pi i(1/2+it)}} \cdot
			\frac{q^{-1/2+it}}{2\pi pv^{1/2+it}} \left(\frac12-it \right) \sum_{0 < |n| < M} \frac{e^{-2q\pi in\alpha/v}}{n^{3/2-it}} 
                        \nonumber \\
        && \qquad + \frac{\Gamma(1/2 - it)}{(2\pi i)^{1/2-it}e^{\pi i(1/2+it)}} 
				\left( \frac{\alpha q}{pv^2} + \frac{(p+q)N}{pv} + \frac{vp}{2q} + \frac{v}2 \right) \nonumber \\
        && \qquad \qquad \times 
                	\left( \frac{q}{v} \right)^{-1/2+it} \sum_{0 < |n| < M} \frac{e^{-2q\pi in\alpha /v}}{n^{1/2-it}}  \nonumber \\
        && \qquad + O(x^{-1/2}).
\end{eqnarray}
The sum of the first term to fourth term on the right-hand side is evaluated as $ \ll |t|^{1/6} \log{|t|} $ 
by using the same method as in (i).
We denote the fifth term and the sixth term on the right-hand side by $ V_1 $ and $ V_2 $, respectively.
By using the Stirling formula, we obtain
\[
		\frac{\Gamma(1/2 - it)}{(2\pi i)^{1/2-it}e^{\pi i(1/2+it)}} = 1 + O(t^{-1})
\]
so 
\[
		|V_1| = \frac{1}{2\pi p \sqrt{qv}} \left| \sum_{0 < |n| < M} \frac{e^{2q \pi in\alpha /v}}{n^{3/2 -it}} \right|  t + O(1).
\]
Also, since $ N \asymp x $ and
\[
	\sum_{0<|n|<M} \frac{e^{-2q\pi i n \alpha /v}}{n^{1/2-it}} \ll x |t|^{1/6} \log{|t|}
\]
is established, then $ V_2 \ll x |t|^{1/6} \log{|t|} $. Therefore, we have
\[
		\zeta_2 \left( \frac12 +it, \alpha; v,w \right) = \kappa t + O(|t|^{1/6} \log{|t|}) 
\]
where $ \kappa = \kappa(t) $ is constant that depends on $ t $ with
\[
		0 < |\kappa| < \frac{1}{2\pi p \sqrt{qv}} \left| \zeta_L \left( \frac32 - it, 1, 1- \frac{q}{v} \alpha  \right) \right|.
\]
\end{enumerate}
\qed

\bigskip

\begin{remark}
The order of $ \zeta_2(1/2 +it, \alpha; v,w) $ greatly different between when $ v, w $ are linearly dependent over $ \Q $
and it is not so. For example, considering the special case of $ v = w = 1 $, since
\[
	\zeta_2(s, \alpha; 1,1) = (1-\alpha) \zeta_H(s, \alpha) + \zeta_H(s-1, \alpha)
\]
holes (See in \cite{SC}, p. 86) and $ \zeta_H (\sigma+it, \alpha) = O(t^{1/2-\sigma}) \ (\sigma < 0) $
is well-known, in the above equation, let $ s = 1/2 + it $ then the second term on the right-hand side is
$ \zeta_H(-1/2 + it, \alpha) = O(t) $. 
From this, we can see that $ \zeta_2(1/2+it, \alpha; 1,1) $ is linear expression of $ t $.
\end{remark}

\bigskip

\textbf{Proof of Theorem \ref{th:Main_Theorem3}}.
Setting $ s = 3/2 + it $ on Theorem \ref{order_of_Barnes_zeta}, then
\begin{eqnarray*}
	&& \zeta_2 \left( \frac32+it, \alpha; v,w \right)
			= \mathop{\sum\sum}\limits_{0 \leq m,n \leq x} 
				(\alpha + vm + wn)^{-3/2-it}		\nonumber \\
	&& \qquad\qquad	+ \frac{(\alpha + vx)^{1/2-it} + (\alpha + wx)^{1/2-it} - (\alpha + vx + wx)^{1/2-it}}
				{-vw(1/2+it)(1/2-it)}  + O(x^{-1/2}).
\end{eqnarray*}
To evaluation for double series on the right-hand side of the above, we consider the following single series 
and by using Lemma \ref{f'''} with same $ g(x) $ and $(a,b)$ in the proof of Theorem \ref{th:Main_Theorem2}(i),
then we have
\begin{eqnarray*}
	\sum_{0 \leq n \leq x} (\alpha+vm+wn)^{-it}
	&=& \sum_{j=1}^\infty \sum_{a < n \leq b} e^{2\pi i g(n)} \nonumber \\
	&\ll &  x (m+x)^{-1/2} |t|^{1/6} + x^{1/2}(m+x)^{1/2}|t|^{-1/6}.
		\label{x,m,t2}
\end{eqnarray*}
Similarly, by using partial summation formula and by using (\ref{x,m,t2}), then
\begin{eqnarray*}
	&& \sum_{0 \leq n \leq x/(v+w)} (\alpha+vm+wn)^{-1/2-it}	\\
	&& \quad \ll  x^{-1/2} (m+x)^{-1/2} |t|^{1/6} + x^{-1}|t|^{-1/6} (m+x)^{1/2}  + x^{-1/2} |t|^{1/6} (m+x)^{1/2} m^{-1} 	\\
	&& \qquad - |t|^{-1/6} \left( \frac{\sqrt{m+x}}{x} \frac{1}{\sqrt{m}} +
				\left\{ \log{\left( \sqrt{1+\frac{x}{m}} -1 \right)} - \log{\left( \sqrt{1+\frac{x}{m}}+1 \right)} \right\} \right).
\end{eqnarray*}
Furthermore, by calculating the sum on $ m $, so we can evaluated this series as follows,
\[
	\mathop{\sum\sum}\limits_{0 \leq m,n \leq x} (\alpha+wm+wn)^{-3/2 - it} \ll |t|^{1/6} + x^{1/2} |t|^{-1/6}.
\]
Therefore, we have
\[
	\zeta_2 \left( \frac32+it, \alpha; v,w \right)
	\ll |t|^{1/6} + x^{1/2} |t|^{-1/6} + x^{1/2} |t|^{-2} + x^{-1/2}.
\]
Taking $ x \asymp |t| $ then $ \zeta_2 \left( 3/2+it, \alpha; v,w \right) \ll |t|^{1/3} $,
 so proof of Theorem \ref{th:Main_Theorem3} is complete.
\qed

\bigskip

\begin{remark}
In the above proof by used Theorem \ref{order_of_Barnes_zeta},
but if we use Theorem \ref{th:Main_Theorem1}, we have only weak results of
$ \zeta_2 \left(3/2+it, \alpha; v,w \right) \ll |t|^{2/3} $.
From this, Theorem \ref{order_of_Barnes_zeta} is more effective when
$ s = 3/2 + it $.
\end{remark}

\section*{Acknowledgments}
First of all, I would like to express my deepest gratitude to my academic advisor
Prof. Kohji Matsumoto for his valuable advice and guidance. 
I also sincerely thank Prof. Takashi Nakamura, Prof. Yayoi Nakamura and Mr. Yuta Suzuki
for their valuable advice and for lots of useful conversations.


\bigskip 
\author{Takashi Miyagawa}:	\\
Faculty of Economy and Information Science, \\
Onomichi City University, \\
1600 Hisayamada-cho Onomichi, 722-8506, Japan \\
E-mail: miyagawa@onomichi-u.ac.jp


\end{document}